\documentclass{amsart}
\usepackage{amsthm,amscd,amssymb,verbatim,epsf,amsmath,amsfonts,mathrsfs,graphicx}
\usepackage[colorlinks=true,linkcolor=blue,citecolor=blue]{hyperref}

\begin{document}
\theoremstyle{plain}
\newtheorem{Thm}{Theorem}
\newtheorem{Cor}{Corollary}
\newtheorem{Con}{Conjecture}
\newtheorem{Main}{Main Theorem}
\newtheorem{Lem}{Lemma}
\newtheorem{Prop}{Proposition}
\theoremstyle{definition}
\newtheorem{Def}{Definition}
\newtheorem{Note}{Note}
\newtheorem{Ex}{Example}
\theoremstyle{remark}
\newtheorem{notation}{Notation}
\renewcommand{\thenotation}{}
\errorcontextlines=0
\numberwithin{equation}{section}
\renewcommand{\rm}{\normalshape}%

\title[Integer Linear Hopf Flows]%
   {Evolving to Non-round Weingarten Spheres: \\ Integer Linear Hopf Flows}

\author{Brendan Guilfoyle}
\address{Brendan Guilfoyle\\
          School of STEM \\
          Institute of Technology, Tralee \\
          Clash \\
          Tralee  \\
          Co. Kerry \\
          Ireland.}
\email{brendan.guilfoyle@ittralee.ie}
\author{Wilhelm Klingenberg}
\address{Wilhelm Klingenberg\\
 Department of Mathematical Sciences\\
 University of Durham\\
 Durham DH1 3LE\\
 United Kingdom.}
\email{wilhelm.klingenberg@durham.ac.uk }

\begin{abstract}
In the 1950's Hopf gave examples of non-round convex 2-spheres in Euclidean 3-space with rotational symmetry that satisfy a linear relationship between their principal curvatures. 

In this paper we investigate conditions under which evolving a smooth rotationally symmetric sphere by a linear combination of its radii of curvature yields a Hopf sphere. When the coefficients of the flow have certain integer values, we show that the fate of an initial sphere is entirely determined by the local geometry of its isolated umbilic points.

A surprising variety of behaviours is uncovered: convergence to round spheres, convergence to non-round Hopf spheres, as well as divergence to infinity. The latter generally involves the sphere passing through its focal set and forming cusps as it goes to infinity.

The critical quantity is the rate of vanishing of the astigmatism - the difference of the radii of curvature - as one approaches the isolated umbilic points of the initial surface. It is proven that the relative size of this rate and the coefficient in the flow function completely determines the fate of the evolution.

This is proven by  reducing the problem to a single linear second order partial differential equation for the astigmatism. The equation is of reaction-diffusion type with convection, and  is solved for arbitrary smooth initial data in terms of associated Legendre polynomials. A dynamic form of quadratures is then used to reconstruct the evolving surface in Euclidean 3-space. 

The geometric setting for the equation is Radius of Curvature space, viewed as a pair of  hyperbolic/Anti-deSitter half-planes joined along their boundary, the umbilic horizon. A rotationally symmetric sphere determines a parameterized curve in this plane with end-points on the umbilic horizon. Any curvature flow then yields a flow of this parameterized curve.

The slope of the curve at the umbilic horizon is linked by the Codazzi-Mainardi equations to the rate of vanishing of astigmatism, and for generic initial conditions can be used to determine the outcome of the flow.

If the slope at the umbilic horizon of the target Hopf sphere is greater than that of the initial sphere, then the flow converges to a round sphere, while if the slope of the target is less than that of the initial sphere the flow diverges. It is when the slopes at the umbilic horizon of initial and target spheres are equal that we get convergence to non-round Hopf spheres. 

The slope can jump during the flow, and a number of examples are given: instant jumps of the initial slope, as well as umbilic circles that  contract to an isolated umbilic point in finite time and 'pop' the slope. 

Finally, we present soliton-like solutions: curves that evolve under linear flows by mutual hyperbolic/AdS isometries (dilation and translation) of Radius of Curvature space.

A forthcoming paper will apply these geometric ideas to non-linear curvature flows. 
 
\end{abstract}

\date{2nd August 2017}
\maketitle
\newpage
\tableofcontents


\section{Introduction and Results}\label{s:1}


A surface $S\subset{\mathbb R}^3$ is said to be {\em Weingarten} if there exists a functional relationship between the radii of curvature $r_1, r_2$
of the surface. Weingarten surfaces, first introduced in 1863 \cite{wein}, form an attractive class of objects which have been much scrutinized  classically and from a contemporary perspective.  The literature, being vast, is cited here somewhat haphazardly.

Well-known examples of surfaces with prescribed Weingarten relation are minimal surfaces, CMC-surfaces, pseudo-spheres and round spheres - see for example \cite{chern1} \cite{chern2} \cite{HandW1} \cite{kuehnel} \cite{lopez} \cite{Voss}. They provide examples of integrable systems \cite{BandM}, arise as stationary solutions of well-known parabolic evolution equations \cite{band} \cite{huisken} and have found application in CAD \cite{vbg}. 

The possible forms that such a relationship can take touches upon topology and analysis at a deep level and is rife with unanswered questions. In this paper we fix the topological setting by considering only 2-spheres, but many of our results would not hold for higher genus, where new phenomena arise. 

We can also make an assumption of convexity on the sphere, as, due to the boundedness of the focal set, convexity can always be achieved by passing to a parallel sphere. This passage induces a simple transformation on a Weingarten relation that is geometric in a certain sense, to be explained below.

The broad purpose of this work is the investigation of the possible Weingarten relations that can arise for surfaces and, in particular, non-round 2-spheres. The interest in non-roundness is due to the fact that, up to an additive constant, any relation can be achieved by the round sphere of a suitable radius. 

The central idea is to take an initial sphere and to deform it using a curvature flow, a second order partial differential equation, in the hope of converging to a sphere with a prescribed relation. Analysis first enters through the type of this partial differential equation - whether the Weingarten relation is elliptic or hyperbolic. 

Unfortunately, an important result of Hartman and Wintner \cite{HandW1} states that, if a sphere satisfies an elliptic Weingarten relation at an umbilic point, it must be round. As a consequence, parabolic curvature flows of surfaces, such as mean curvature flow or Gauss curvature flow, can, at best, hope to yield round 2-spheres.

The role of the isolated umbilic points in this result is worthy of careful consideration and is not accidental. This paper is set in a geometrized curvature (2-jet) space of classical surface theory, in a similar way to the proof of the Carath\'{e}odory Conjecture takes place in the geometrized 1-jet space, oriented line space \cite{gak0}. The centrality of isolated umbilic points in this new setting underscores once again their beautiful subtlety.

The key to making progress on flowing to non-round spheres is to note that the rotationally symmetric reduction of a hyperbolic relation can be used to construct a parabolic curvature flow, with only one space dimension. Since curvature flows preserve rotational symmetry, a parabolic flow can therefore find solutions of a hyperbolic equation.

Moreover, a theorem of Voss \cite{Voss} states that any sphere satisfying an analytic Weingarten relation must be rotationally symmetric. The assumption of real analyticity here is critical in the proof. It is not known whether there exist smooth Weingarten spheres which are {\it not} rotationally symmetric.

In the rotationally symmetric case, the distinction between smooth and real analyticity condenses on the isolated umbilic points. This paper develops a dynamic perspective on this and shows how the the isolated umbilic points can also control the behaviour of a curvature flow.

The evolution considered is by a linear Weingarten relation. A later paper \cite{gak8} will consider non-linear evolution. The simplifications in the linear case allow one to completely integrate the flow, yet exhibits surprisingly complicated behaviour. Moreover, for a linear Weingarten relation there are non-round targets for the flow, namely Hopf spheres.

The Hopf spheres are a 2-parameter family of rotationally symmetric, non-round, strictly convex spheres that satisfy a linear relationship on the principal curvatures \cite{Hopf}. The relation they satisfy is a second order partial differential equation which is hyperbolic, and therefore side-steps Hartman and Wintner's obstacle to non-roundness. 

They are all at least $C^2$ smooth but are real analytic only for certain discrete values of the parameters. Hopf spheres which also satisfy a linear relation on the {\it radii} of curvature form a 1-parameter sub-family we call {\it linear Hopf spheres}. In fact, the term will include spheres parallel to such Hopf spheres - these are the stationary limits to which the flow hopefully converges.

In more detail, consider evolving an initial smooth rotationally symmetric convex sphere $S_0\subset{\mathbb R}^3$ to a family of rotationally symmetric spheres  $\vec{X}:S^2\times[0,t)\rightarrow {\mathbb R}^3$ such that
\begin{equation}\label{e:hopfflow1}
\frac{\partial \vec{X}}{\partial t}^\perp =-{\mathbb K}\;\vec{N}, \qquad\qquad \vec{X}(S^2,0)=S_0,
\end{equation}
where
\begin{equation}\label{e:hopfflow2}
{\mathbb K}={\textstyle{\frac{1}{2}}}(r_1+r_2)+{\textstyle{\frac{\lambda}{2}}}(r_1-r_2)-\psi_\infty,
\end{equation}
for constants $\lambda>1$ and $\psi_\infty>0$, and $\vec{N}$ is the unit normal to the sphere. Here and throughout $r_2$ is the radius of curvature of the profile curve, while $r_1$ is the other radius of curvature. 

The difference of the radii of curvature is a critical quantity in what follows as it vanishes at, and only at, umbilic points. It is often referred to as the astigmatism of the surface \cite{MandP}. The behaviour of the flow is entirely determined by the vanishing of the astigmatism at the isolated umbilic points of the initial surface. 

In particular, a rotationally symmetric sphere is said to be {\it of order k} if the astigmatism vanishes like
the $2+2k$th power of the polar angle at the umbilic points (full definitions are in Section \ref{s:2.2}). The order of a generic smooth surface is 0.

We prove the following:

\vspace{0.1in}
\begin{Thm}\label{t:1}
Let $S_0$ be a smooth embedded convex rotationally symmetric sphere of order $k$.
The linear Hopf flow (\ref{e:hopfflow1}) and (\ref{e:hopfflow2}) with 
\begin{equation}\label{e:hopfflow3}
\lambda=1+\frac{1}{n+1} \qquad{\mbox{for }}n\in{\mathbb N},
\end{equation}
behaves in the following manner:
\begin{itemize}
\item[(1)] if $n<k$, the evolving sphere converges exponentially through smooth convex spheres to the round sphere of radius $\psi_\infty$,
\item[(2)] if $n=k$, an initial non-round sphere converges exponentially through smooth convex spheres to a non-round linear Hopf sphere,
\item[(3)] if $n>k$, the sphere diverges exponentially.
\end{itemize}
In the last case, the sphere in ${\mathbb R}^3$ may pass through its focal set as it expands to infinity. 
\end{Thm}

\vspace{0.1in}

The linear Hopf flow (\ref{e:hopfflow1}), (\ref{e:hopfflow2})  satisfying (\ref{e:hopfflow3}) is referred to as {\it integer} linear Hopf flow. The quantized nature of the slope at the umbilic points - and hence its expression in the flowing relation - is a further manifestation of the deep role of umbilic points in the analysis of the surface.

To prove this Theorem we view the flow in radii of curvature (RoC) space, which for convex surfaces is an open cone in ${\mathbb R}^2$ \cite{gak7}. 
The initial surface is rotationally symmetric and remains so under the flow, and therefore its image in RoC space (its RoC diagram) is a curve. For example, the RoC diagram of a linear Hopf spheres is a line segment with end-point on the horizontal axis. 

We derive the evolution equations of this curve in RoC space under a linear curvature flow (\ref{e:hopfflow1}) and (\ref{e:hopfflow2}). In full generality, this is a system of second order partial differential equations with mixed Dirichlet and Neumann boundary conditions, but it can be reduced to a single linear reaction-diffusion equation with convection for the astigmatism with Dirichlet boundary conditions. 

This PDE for the astigmatism is then solved in closed form for smooth initial data and integer values of the flowing function. The full curve flow is reconstructed by quadratures and we present the explicit solution in terms of the initial data in Theorems \ref{t:5} and \ref{t:6} of Section \ref{s:3.2}.

The order of a sphere is closely related to the slope at the isolated umbilic points. That is, at the north pole $N$ and south pole $S$, define the 
{\it slopes at the umbilic points} to be the slopes $\mu_N,\mu_S$ of its RoC diagram at the North and South poles. These quantities play a key role
in the work of Hopf \cite{Hopf} and can only take on certain values for smooth surfaces.

Generically, the slopes at the umbilic points are both equal to 2.  For a surface of order $k$ the slopes satisfy
\[
\mu_N,\mu_S\leq 1+\frac{1}{k+1},
\] 
with equality for generic surfaces of order $k$ - which we refer to as the non-degenerate case. This is proven in Theorem \ref{t:4} of Section \ref{s:2.4}. For linear Hopf spheres satisfying ${\mathbb K}=0$ with (\ref{e:hopfflow2}) holding, clearly the slope at the umbilic points is $\mu=\lambda$.

Our result can be recast in terms of slopes:
\vspace{0.1in}
\begin{Thm}\label{t:2}
Let $S_0$ be a smooth strictly convex rotationally symmetric non-degenerate surface with slope at the umbilic points $\mu=\mu_N=\mu_S$.

Under integer linear Hopf flow (\ref{e:hopfflow1}), (\ref{e:hopfflow2}) and (\ref{e:hopfflow3}), the evolving sphere
behaves in the following manner:
\begin{itemize}
\item[(1)] if $\lambda>\mu$,the evolving sphere converges exponentially through smooth convex spheres to the round sphere of radius $\psi_\infty$,
\item[(2)] if $\lambda=\mu$, an initial non-round sphere converges exponentially through smooth convex spheres to a non-round linear Hopf sphere.
\item[(3)] if $1<\lambda<\mu$, the sphere diverges exponentially.
\end{itemize}
\end{Thm}

\vspace{0.1in}
We also extract the following:

\vspace{0.1in}
\begin{Thm}\label{t:3}
Under the linear Hopf flow (\ref{e:hopfflow1}) and (\ref{e:hopfflow2}) with $\lambda=2$ any smooth convex rotationally symmetric sphere $S_0$ converges exponentially through smooth spheres to a linear Hopf sphere, which is non-round if $S_0$ is of order 0 and round otherwise.
\end{Thm}
\vspace{0.1in}

For non-integer linear Hopf flows the situation is less clear-cut. In particular, the stationary solutions to which the flow should converge are not smooth and so there can be a loss of differentiability. Explicit exact solutions are still easy to construct, although they are not smooth and involve Legendre functions (rather than polynomials). This presents difficulties to writing the general solution in terms of a complete set of modes. 

In this regard, Proposition \ref{p:soliton} is the only significant result in this paper that applies to non-integer flows, although there is every reason to believe that Theorem \ref{t:2} holds for non-integer values of $\lambda$.

The next section contains background material on rotationally symmetric surfaces and their RoC diagrams. This fixes notation and introduces the main examples and tools. Further details on RoC space in general can be found in \cite{gak7}. Section \ref{s:3} presents the initial value problem for  linear Hopf flow and its solution in the integer case. Theorems \ref{t:1}, \ref{t:2} and \ref{t:3} then follow.

The final section considers the properties of these solutions. A family of examples of sufficient generality is constructed to directly illustrate various behaviour: divergence, convergence to round and non-round linear Hopf spheres, slope jumping and the contraction of umbilic circles to the poles.

We also present solitons: spheres that evolve by a linear curvature flow such that the induced motion of its RoC diagram is by an isometry of the hyperbolic/AdS metric.

In a forthcoming paper we will extend this approach to non-linear curvature flows \cite{gak8}
\vspace{0.1in}


\section{Radius of Curvature Space}\label{s:2}


This section introduces the basic differential geometric tools needed from surface theory. The setting is Radius of Curvature space which is a geometrized 2-jet. Going from RoC space to Euclidean 3-space is explained and the Hopf spheres are discussed. These are the stationary limits to which we ultimately hope the flow converges.

\vspace{0.1in}

\subsection{Classical surface theory}\label{s:2.1}

Consider a smooth, closed, convex sphere $S\subset{\mathbb R}^3$ with support function $r:S\rightarrow{\mathbb R}$. While most of the
discussion applies locally to any surface away from flat points, throughout we usually have that $S$ is a convex sphere. 

Moving a surface parallel along its outward pointing oriented normal lines corresponds to adding a constant to the support function $r$. Shrinking a  convex surface along its oriented normal lines eventually leads to a loss of convexity - points where the surface touches its focal set \cite{gak3}. Continuing to shrink, the surface pulls itself inside out and one obtains a closed convex sphere with inward pointing normal line. 

Most of the behaviour discussed in this paper is at a higher derivative level, and so are invariant under such parallel changes. Thus suitable modifications can be made to ensure that strict convexity is maintained, and it is also easy to describe the singularities that can develop.

The simplest examples of Weingarten surfaces are those with rotational symmetry:
\vspace{0.1in}

\begin{Def}
A surface $S$ is {\em rotationally symmetric} if $S$ is invariant under a subgroup $S^1\subset{\mbox{Euc}}({\mathbb R}^3)$, so that for 
$Rot_\Theta\in S^1$ we have $Rot_\Theta(S)=S$. 

For a closed convex sphere this action has two fixed points, both umbilic points, which are antipodal under the Gauss map.
\end{Def}
\vspace{0.1in}

Parameterize $S$ by the inverse of the Gauss map, with standard polar coordinates $(\theta,\phi)$. For a rotationally symmetric surface with fixed points at $\theta=0,\pi$, the support function and curvatures are functions only of $\theta$.

\begin{Prop}\cite{gak4}
If $r_1$ and $r_2$ are the radii of curvature of a rotationally symmetric sphere $S$, with $r_2$ the radius of curvature of the profile curve, 
then in Gauss polar coordinates  for $0\leq\theta\leq\pi$,  
$0\leq\phi\leq 2\pi$ we find that
\begin{equation}\label{e:psi_def}
\psi={\textstyle{\frac{1}{2}}}(r_1+r_2)={\textstyle{\frac{1}{2}}}\frac{d^2 r}{d\theta^2}+{\textstyle{\frac{1}{2}}}\cot\theta\frac{d r}{d\theta}+r
    =r+\frac{1}{2\sin\theta}\frac{d}{d\theta}\left(\sin\theta\frac{dr}{d\theta}\right)
\end{equation}
\begin{equation}\label{e:s_def}
s={\textstyle{\frac{1}{2}}}(r_1-r_2)=-{\textstyle{\frac{1}{2}}}\frac{d^2 r}{d\theta^2}+{\textstyle{\frac{1}{2}}}\cot\theta\frac{d r}{d\theta}
   =-{\textstyle{\frac{1}{2}}}\sin\theta\frac{d}{d\theta}\left(\frac{1}{\sin\theta}\frac{dr}{d\theta}\right).
\end{equation}
\end{Prop}
\vspace{0.1in}

\begin{Def}
A convex rotationally symmetric sphere is called {\it oblate} if $r_1\geq r_2$ everywhere, and {\it prolate} if $r_1\leq r_2$ everywhere.
Thus the function $ s$ is positive if the sphere is oblate and negative if it is prolate. 
\end{Def}
\vspace{0.1in}

\begin{Prop}
On a rotationally symmetric surface, the Codazzi-Mainardi equation identities hold:
\begin{equation}\label{e:comain_rs}
\frac{d }{d\theta}(\psi+ s )+2\cot\theta \;s=0.
\end{equation}
\end{Prop}
\begin{proof}
This can be seen by checking that equations (\ref{e:psi_def}) and (\ref{e:s_def}) imply equation (\ref{e:comain_rs}).
\end{proof}

\vspace{0.1in}
Equation (\ref{e:comain_rs}) is necessary and sufficient for a map to arise as the RoC diagram of a rotationally symmetric surface:

\vspace{0.1in}
\begin{Prop}\label{p:cmint}
For every proper $C^1-$smooth curve  
${\mathcal C}([0,\pi],\{0,\pi\})\rightarrow({\mathbb R}^2, {\mathbb R}):\theta\mapsto(\psi(\theta),s(\theta))$,
that satisfies the Codazzi-Mainardi equation (\ref{e:comain_rs}) for $0\leq\theta\leq\pi$,
there exists a closed rotationally symmetric $C^3-$smooth surface $S$ in ${\mathbb R}^3$ in Gauss coordinates $(\theta,\phi)$, 
such that the curve ${\mathcal C}$ is the image of the RoC diagram of $S$. 

The sphere $S$ in ${\mathbb R}^3$ is unique up to translation parallel to the axis of symmetry. 
\end{Prop}
\begin{proof}
Given a curve $(\psi(\theta),s(\theta))$ satisfying (\ref{e:comain_rs}), define the function $r:[0,\pi]\rightarrow{\mathbb R}$ by
\[
{\textstyle{\frac{1}{2}}}\sin\theta\frac{d}{d\theta}\left(\frac{1}{\sin\theta}\frac{dr}{d\theta}\right)=- s .
\]
That is
\begin{equation}\label{e:supp_wein}
r=C_2-C_1\cos\theta-2\int\left[\sin\theta\int\frac{s}{\sin\theta}d\theta\right]d\theta.
\end{equation}

Thus $r$ is the support function of a closed rotationally symmetric surface whose RoC diagram is the given curve. 

The value of the constant $C_1$ is changed by parallel translation along the axis of symmetry, $C_2$ is fixed by the value of $\psi$.
\end{proof}

\vspace{0.1in}
\subsection{Quadratures and degeneracy}\label{s:2.2}

Given the function $s(\theta)$, equation (\ref{e:supp_wein}) allows one to find the support function $r(\theta)$ ``up to quadratures'' i.e. one's
ability to carry out the double integration explicitly. In the following Proposition we carry out the process explicitly for certain classes of functions $s(\theta)$ that we use in section \ref{s:3.2}.

In particular, introduce the decomposition of $s$ in terms of trigonometric polynomials:

\begin{Prop}\label{p:decomp}
For any smooth convex sphere $S$, the difference of the radii of curvature can be written
\[
s=\sum_{l=0}^\infty(a_l+b_l\cos\theta)\sin^{2l+2}\theta,
\]
for constants $a_l,b_l$.
\end{Prop}
\begin{proof}
Standard.
\end{proof}
\vspace{0.1in}

For each of these two terms we carry out quadratures.

\vspace{0.1in}
\begin{Prop}\label{p:quadsex1}
If $s=C_0\sin^{2l+2}\theta$ then 
\begin{equation}\label{e:sinquad1}
r=C_2+C_1\cos\theta-2C_0\sum\limits_{k=0}^l\frac{(-1)^k}{(2k+1)(2k+2)}\binom{l}{k}\cos^{2k+2}\theta,
\end{equation}
and 
\begin{equation}\label{e:sinquad2}
\psi=C_2+\sum_{k=0}^l(-1)^kC_0\frac{k+2}{k+1}\binom{l}{k}-\left(1+\frac{1}{l+1}\right)C_0\sin^{2l+2}\theta.
\end{equation}

If $s=C_0\cos\theta\sin^{2l+2}\theta$ then 
\begin{equation}\label{e:cossinquad1}
r=C_2+C_1\cos\theta-\sum_{k=0}^{l+1}C_0\frac{(-1)^k}{(l+1)(2k+1)}\binom{l+1}{k}\cos^{2k+1}\theta,
\end{equation}
and 
\begin{equation}\label{e:cossinquad2}
\psi=C_2+\sum_{k=0}^l\sum_{m=0}^kC_0\frac{(-1)^{k+m}[1-(2k+1)(m+1)]}{(l+1)(2k+1)}\binom{l+1}{k}\binom{k}{m}\cos\theta\sin^{2m}\theta.
\end{equation}
In addition, we have
\begin{equation}\label{e:cossinquad3}
\frac{d \psi}{d\theta}=-2C_0(l+2)\sin^{2l+1}\theta+C_0(2l+5)\sin^{2l+3}\theta.
\end{equation}
\end{Prop}
\begin{proof}
For ease of notation drop the constants of integration term $C_2+C_1\cos\theta$ for $r$.

To prove equation (\ref{e:sinquad1}), use quadratures on the expression for $s$
\begin{align}
r=&-2\int\sin\theta d\theta\int\frac{s}{\sin\theta}d\theta=-2C_0\int d\cos\theta\int(1-\cos^2\theta)^{2l} d\cos\theta\nonumber\\
&=-2C_0\sum_{k=0}^l\frac{(-1)^k}{(2k+1)(2k+2)}\binom{l}{k}\cos^{2k+2}\theta\nonumber.
\end{align}

To prove equation (\ref{e:sinquad2}), differentiate this expression for $r$ using equation (\ref{e:psi_def}). The result, after some simplification is
\[
\psi=C_0\left[-\sin^{2l}\theta+\sum_{k=0}^l\binom{l}{k}\frac{k+2}{k+1}+\sum_{m=1}^{l+1}\sum_{k=m-1}^l(-1)^{k+m}\binom{l}{k}\binom{k+1}{m}\frac{k+2}{k+1}\sin^{2m}\theta \right].
\]

Now equation (\ref{e:sinquad2}) follows from:

\vspace{0.1in}
\begin{Lem}
For all $l,m\in{\mathbb N}$ with $l\geq m-1\geq0$ the following identities hold
\[
\sum_{k=m-1}^l(-1)^{k+m}\frac{k+2}{k+1}\binom{l}{k}\binom{k+1}{m}=
\left\{\begin{array}{ccl}
                  0&if& l>m\\
                  1&if& l=m\\
                  -1-\frac{1}{l+1} &if& l=m-1
                \end{array}.
              \right.
\]
\end{Lem}
\begin{proof}

This can be proven as follows: for all $l,m\in{\mathbb N}$ with $0\leq m-1\leq l$,
\begin{align}
{\sum_{k=m-1}^{l}}&{(-1)^{k+m}\frac{k+2}{k+1}\binom{l}{k}\binom{k+1}{m}}=\frac{1}{m}\sum_{k=m-1}^l(-1)^{k+m}(k+2)\binom{l}{k}\binom{k}{m-1}\nonumber\\
&=\frac{1}{m}\binom{l}{m-1}\sum_{k=m-1}^l(-1)^{k+m}(k+2)\binom{l-m+1}{k-m+1}\nonumber\\
&=\frac{1}{m}\binom{l}{m-1}\sum_{k=0}^{l-m+1}(-1)^{k+1}(k+m+1)\binom{l-m+1}{k}\nonumber\\
&=\frac{m+1}{m}\binom{l}{m-1}\sum_{k=0}^{l-m+1}(-1)^{k+1}\binom{l-m+1}{k}\nonumber\\
&\qquad+\frac{l-m+1}{m}\binom{l}{m-1}\sum_{k=1}^{l-m+1}(-1)^{k+1}\binom{l-m}{k-1}\nonumber\\
&=-\frac{m}{m+1}\binom{l}{m-1}[[l=m-1]]\nonumber\\
&\qquad+\frac{l-m+1}{m}\binom{l}{m-1}\sum_{k=0}^{l-m}(-1)^k\binom{l-m}{k}\nonumber\\
&=\left(-1-\frac{1}{m}\right)[[l=m-1]]+\frac{l-m+1}{m}\binom{l}{m-1}[[l=m]]\nonumber\\
&{=\left(-1-\frac{1}{m}\right)[[l=m-1]]+1[[l=m]]}.\nonumber
\end{align}
Thanks to Markus Scheuer for pointing out this proof (see \href{https://tinyurl.com/y7po8jnl}{this MathOverflow discussion for details}).
\end{proof}
\vspace{0.1in}

To prove equation (\ref{e:cossinquad1}), use quadratures on the expression for $s$, while equation (\ref{e:cossinquad2}) follows from the definition of $\psi$, as before.

Finally, differentiating  equation (\ref{e:cossinquad2}) yields
\begin{align}
\frac{d\psi}{d\theta}=& (2l+5)\sin^{2l+3}\theta+{\textstyle{\frac{1}{l+1}}}\sum_{m=0}^l\left[\sum\limits_{k=m+1}^{l+1}(-1)^{k+m}\binom{l+1}{k}\left({\textstyle{(1-(2k+1)(m+2))\frac{2m+2}{2k+1}\binom{k}{m+1}}}\right.\right.\nonumber\\
&+\left.{\textstyle{(1-(2k+1)(m+1))\frac{2m+1}{2k+1}\binom{k}{m}}}\right)+\left.{\textstyle{(1-(2m+1)(m+1))\binom{l+1}{m}}}\right]\sin^{2m+1}\theta.\nonumber
\end{align}

Now equation (\ref{e:cossinquad3}) follows from the identity:
\vspace{0.1in}

\begin{Lem}
For all $l,m\in{\mathbb N}$ with $l\geq m\geq0$ the following identities hold
\begin{eqnarray}
&\sum\limits_{k=m+1}^{l+1}(-1)^{k+m}\binom{l+1}{k}\left[(1-(2k+1)(m+2))\frac{2m+2}{2k+1}\binom{k}{m+1}+(1-(2k+1)(m+1))\frac{2m+1}{2k+1}\binom{k}{m}\right]\nonumber\\
&\qquad\qquad\qquad\qquad\qquad+(1-(2m+1)(m+1))\binom{l+1}{m}=
\left\{\begin{array}{ccl}
                  0&if& l>m\\
                  2(l+1)(l+2) &if& l=m
                \end{array}.
              \right.\nonumber
\end{eqnarray}
\end{Lem}
\begin{proof}
This can be proven as follows:
\begin{align}
{\sum_{k=m+1}^{l+1}}&{(-1)^{k+m}\binom{l+1}{k}
\left[(1-(2k+1)(m+2))\frac{2m+2}{2k+1}\binom{k}{m+1}\right.}\nonumber\\
&\qquad\qquad\qquad\qquad\quad
{{\left.+(1-(2k+1)(m+1))\frac{2m+1}{2k+1}\binom{k}{m}\right]}}\nonumber\\
&=\sum_{k=m+1}^{l+1}(-1)^{k+m}\binom{l+1}{k}[m-2k(m+2)]\binom{k}{m}\nonumber\\
&=\binom{l+1}{m}\sum_{k=m+1}^{l+1}(-1)^{k+m}\binom{l+1-m}{k-m}[m-2k(m+2)]\nonumber\\
&=\binom{l+1}{m}\sum_{k=1}^{l+1-m}(-1)^{k}\binom{l+1-m}{k}[-2k(m+2)-m(2m+3)]\nonumber\\
&=-2(m+2)\binom{l+1}{m}\sum_{k=1}^{l+1-m}(-1)^{k}\binom{l+1-m}{k}k\nonumber\\
&\qquad-m(2m+3)\binom{l+1}{m}\left([[l+1=m]]-1\right)\nonumber\\
&=-2(m+2)\binom{l+1}{l+1-m}(l+1-m)\sum_{k=1}^{l+1-m}(-1)^{k}\binom{l-m}{k-1}\nonumber\\
&\qquad-m(2m+3)\binom{l+1}{m}\left([[l+1=m]]-1\right)\nonumber\\
&=2(m+2)(l+1)\binom{l}{m}\sum_{k=0}^{l-m}(-1)^{k}\binom{l-m}{k}\nonumber\\
&\qquad-m(2m+3)\binom{l+1}{m}\left([[l+1=m]]-1\right)\nonumber\\
&{=2(l+1)(l+2)[[l=m]]}\nonumber\\
&\qquad{-(1-(2m+1)(m+1))\binom{l+1}{m}\left([[l+1=m]]-1\right)}.\nonumber
\end{align}
Thanks to Markus Scheuer for pointing out this proof and to Darij Grinberg for giving an alternative proof (see \href{https://mathoverflow.net/questions/278074/is-there-a-simple-proof-of-the-following-identity-part-2}{this MathOverflow discussion for details}).
\end{proof}

\vspace{0.1in}
\end{proof}

\vspace{0.1in}

A further class that will be of use are combinations of associated Legendre polynomials:

\vspace{0.1in}
\begin{Prop}\label{p:quadsex2}
If $s=C_0\sin^{2+n}\theta\; P^n_l(\cos\theta)$ for the associated Legendre polynomial $P^n_l(\cos\theta)$ with $l>n+2$ then 
\begin{equation}\label{e:legpolyquad1}
r=C_2+C_1\cos\theta-C_0\frac{2\sin^{2+n}\theta P^{2+n}_l(\cos\theta)}{(l+n+2)(l+n+1)(l-n)(l-n-1)},
\end{equation}
and 
\begin{align}\label{e:legpolyquad2}
\psi&=C_2-C_0\sin^{n}\theta\left[ \frac{(l+n+2)^2-(l+n+4)(l+n+1)\sin^2\theta}{(l+n+2)(l+n+1)(l-n)(l-n-1)}P^{2+n}_l(\cos\theta)\right.\nonumber\\
&\qquad\left.-\frac{(2l+2n+3)\cos\theta}{(l+n+1)(l-n)(l-n-1)} P^{2+n}_{l-1}(\cos\theta)+\frac{1}{(l-n)(l-n-1)}P^{2+n}_{l-2}(\cos\theta)\right].\nonumber
\end{align}
\end{Prop}
\begin{proof}
For ease of notation again we drop the constants of integration term $C_2+C_1\cos\theta$ for $r$.

To prove equation (\ref{e:legpolyquad1}), use quadratures on the expression for $s$
\begin{align}
r=&-2\int\sin\theta d\theta\int\frac{s}{\sin\theta}d\theta=-2\int d\cos\theta\int\sin^n\theta P^n_l(\cos\theta) d\cos\theta\nonumber\\
&=-2\frac{(-1)^n}{2^l l!}\int dx\int(1-x^2)^n\frac{d^{l+n}}{dx^{l+n}}(x^2-1)^ldx,\nonumber
\end{align}
where $x=\cos\theta$ and we use the usual identities for associated Legendre polynomials:
\[
P^n_l(x)=\frac{(-1)^n}{2^l l!}(1-x^2)^{\scriptstyle{\frac{n}{2}}}\frac{d^{l+n}}{dx^{l+n}}(x^2-1)^l.
\]
Furthermore, recall the identity
\[
\frac{d^{l-n}}{dx^{l-n}}(x^2-1)^l=(-1)^n\frac{(l-n)!}{(l+n)!}(1-x^2)^n\frac{d^{l+n}}{dx^{l+n}}(x^2-1)^l,
\]
which means that
\begin{align}
r=&-2\frac{(l+n)!}{2^l l!(l-n)!}\int dx\int\frac{d^{l-n}}{dx^{l-n}}(x^2-1)^ldx\nonumber\\
&=-2\frac{(l+n)!}{2^l l!(l-n)!}\frac{d^{l-n-2}}{dx^{l-n-2}}(x^2-1)^l\nonumber\\
&=-2\frac{(l+n)!(-1)^n(l-n-2)!}{2^l l!(l-n)!(l+k+2)!}(1-x^2)^{2+n}\frac{d^{l+n+2}}{dx^{l+n+2}}(x^2-1)^l\nonumber\\
&=-\frac{2\sin^{2+n}\theta P^{2+n}_l(\cos\theta)}{(l+k+2)(l+n+1)(l-n)(l-n-1)},\nonumber
\end{align}
as claimed.
\end{proof}

\vspace{0.1in}

\begin{Def}
Given the decomposition of $s$ as in Proposition \ref{p:decomp}, the surface is said to be {\it of order k} if $a_l=b_l=0$ for all $l<k$ and $a_k, b_k$ are not both 
zero. 

A surface of order $k$ is {\it non-degenerate} if moreover $a_k^2\neq b_k^2$.
\end{Def}
\vspace{0.1in}

A generic smooth surface is of order 0, while a generic surface of order $k$ is non-degenerate. A Hopf sphere with $\mu=1+\frac{1}{n+1}$ is
of order $n$.

\subsection{RoC space}\label{s:2.3}
The {\it RoC diagram} of $S$ is the image of the map $f:S^2\rightarrow{\mathbb R}^2$ taking $(\theta,\phi)\in S$ to $(\psi={\textstyle{\frac{1}{2}}}(r_1+r_2),s={\textstyle{\frac{1}{2}}}(r_1-r_2))$. Thus, the umbilic points lie along the horizontal axis. This is equivalent to the curvature diagram discussed by Hopf \cite{Hopf_book} but more suited to our purposes, as elucidated in \cite{gak7}. 

RoC space can be thought of as two half-planes joined at the umbilic horizon. The hyperbolic/Anti-deSitter area 2-form on these half-planes pulled back to $S$  is equal to the curvature 2-form of the induced Lorentz metric on the set of oriented normal lines to $S$, viewed as a surface in $T{\mathbb S}^2$, pulled back to $S$ \cite{gak7}. 

Hyperbolic/AdS dilations and translations of the RoC diagram of $S$ come from scaling about the origin and addition of a constant to the 
support function of $S$ in ${\mathbb R}^3$, respectively.

Not every such parameterized map arises as the RoC diagram of a surface in ${\mathbb R}^3$, as the Codazzi-Mainardi conditions (\ref{e:comain_rs})
must hold.

The RoC diagram of a rotationally symmetric sphere is a curve in the plane and if the sphere is convex, it lies in the open cone $\{\psi+s>0\}\cap\{\psi-s>0\}$. Moving the diagram in RoC space by a translation parallel to the umbilic horizon moves the sphere in  ${\mathbb R}^3$ to a parallel surface. If shrunk enough (i.e. moved to the left in RoC space) it will eventually hit the diagonal and lose convexity at its focal set \cite{gak3}.

Continue to push the sphere through its focal set and it will develop cusps and self-intersections which eventually resolve themselves as one pushes all the way through the diagonals. The result is a convex sphere parameterised by its inward-pointing normal. This well-known behaviour of the quadratic complex along oriented lines in ${\mathbb R}^3$ is represented in RoC space by horizontal translation. 

To get a better picture of the RoC diagram of a rotationally symmetric sphere, consider the simple linear combination:
\[
s=(a_0+b_0\cos\theta)\sin^2\theta+(a_1+b_1\cos\theta)\sin^4\theta,
\]
which, by Propositions \ref{p:quadsex1} and \ref{p:quadsex2}, has mean radii of curvature
\begin{align}
\psi=&\psi_\infty+(c_0+({\textstyle{\frac{2}{3}}}b_0+{\textstyle{\frac{4}{15}}}b_1)(\cos\theta-1))
+[-2a_0+(-{\textstyle{\frac{5}{3}}}b_0+{\textstyle{\frac{2}{15}}}b_1)\cos\theta]
\sin^2\theta\nonumber\\
&\qquad\qquad+{\textstyle{\frac{1}{10}}}[15a_1+14b_1\cos\theta]\sin^4\theta,\nonumber
\end{align}
for constants $\psi_\infty$ and $c_0$.

In Figure 1 the RoC diagrams are given for three different spheres: 
\begin{itemize}
\item[1A]: $\psi_\infty=10,c_0=1,a_0=3,a_1=10,b_0=2,b_1=7$, 
\item[1B]: $\psi_\infty=10,c_0=1,a_0=3,a_1=10,b_0=20,b_1=7$,
\item[1C]: $\psi_\infty=1,c_0=1,a_0=3,a_1=5,b_0=-36,b_1=50$.
\end{itemize}
The red horizontal line is the umbilic horizon.

Figure 1A shows the simplest oblate rotationally symmetric surface, while Figure 1B is the RoC diagram of a turnip-shaped sphere: a prolate and an 
oblate disc joined along an umbilic circle. Figure 1C demonstrates how the RoC diagram does not have to be an embedding.

\vspace{0.1in}

\subsection{The slope at an umbilic}\label{s:2.4}
\vspace{0.1in}

For any smooth rotationally symmetric sphere with the north pole $N$ and south pole $S$, define the {\it slopes at the isolated umbilic points} to be the slopes of the RoC diagram at the north and south poles:
\[
\mu_N=\lim_{p\rightarrow N}\frac{\psi(p)-\psi(N)}{s(p)} \qquad\qquad \mu_S=\lim_{p\rightarrow S}\frac{\psi(p)-\psi(S)}{s(p)}.
\]
In other words, the slopes of the endpoints on the umbilic horizon. We prove the following:

\begin{Thm}\label{t:4}
For a smooth rotationally symmetric sphere of order $k$ the slopes at the isolated umbilic points satisfy
\[
\mu_N,\mu_S\leq 1+\frac{1}{k+1}.
\]
with equality when the sphere is non-degenerate.
\end{Thm}

\begin{proof}
Starting with the decomposition of $s$ given in Proposition \ref{p:decomp} 
\[
s=\sum_{l=0}^\infty(a_l+b_l\cos\theta)\sin^{2l+2}\theta,
\]
for constants $a_l,b_l$. Suppose that the sphere is of order $k$ so that we write
\[
s=(a_k+b_k\cos\theta)\sin^{2k+2}\theta+O(2k+4),
\]
for $a_k,b_k$ not both zero. These can be integrated directly for $\psi$ using Proposition \ref{p:quadsex1} and differentiated to find that
\[
\frac{d s}{d\theta}=2(k+1)(b_k+a_k\cos\theta)\sin^{2k+1}\theta+O(2k+3)
\]
and
\[
\frac{d \psi}{d\theta}=-2(k+2)(b_k+a_k\cos\theta)\sin^{2k+1}\theta+O(2k+3),
\]
and therefore
\[
\mu_{N}=\lim_{\theta\rightarrow0}\frac{\psi(0)-\psi(\theta)}{s(\theta)}=-\lim_{\theta\rightarrow0}\frac{\frac{d \psi}{d\theta}}{\frac{d s}{d\theta}}=\lim_{\theta\rightarrow0}\frac{(k+2)(b_k+a_k\cos\theta)\sin^{2k+1}\theta+O(2k+3)}{(k+1)(b_k+a_k\cos\theta)\sin^{2k+1}\theta+O(2k+3)}
\]
If $a_k\neq \pm b_k$, i.e. the surface is non-degenerate, then this limit is equal to 
\[
\mu_{N}=1+\frac{1}{k+1}
\]
and the same for $\mu_S$. 

In the degenerate case, the limit is computed by taking higher derivatives, which increases the order and decreases the slope, hence the inequality. 

\end{proof}
\vspace{0.1in}

\subsection{Linear Hopf spheres}\label{s:2.5}

Consider the case of a rotationally symmetric surface satisfying the linear relation
\begin{equation}\label{e:lihsdef}
\psi+\mu s =\psi_\infty, \qquad\qquad {\mbox{for constants }}\psi_\infty>0,\mu>1.
\end{equation}
Note that a relation of the form $a\kappa + bH = c$, for Gauss curvature $\kappa$ and mean curvature $H$, is called "linear" in the literature 
\cite{lopez}, and Hopf's examples, also satisfy a linear relation, but between the curvatures \cite{Hopf}. In contrast, linear Hopf spheres
for us will refer to the relationship between the radii of curvature, as above.

Note also that moving to a parallel surface changes $\psi_\infty$ by an additive constant.

Combining the derivative  w.r.t. $\theta$ of this linear relationship with the Codazzi-Mainardi relation (\ref{e:comain_rs}) yields
\[
\frac{1-\mu}{s}\frac{d s }{d\theta}=-2\cot\theta,
\]
which integrates to
\[
 s =C_0\sin^{\scriptstyle{\frac{2}{\mu-1}}}\theta.
\]
This can be integrated directly for the support function, using equation (\ref{e:supp_wein}), 
yielding a family of smooth spheres with:
\[
r=C_2+C_1\cos\theta-2C_0\int\left[\sin\theta\int\sin^{\scriptstyle{\frac{3-\mu}{\mu-1}}}\theta d\theta\right]d\theta, 
\]
which are non-round for $C_0\neq0$, convex for large enough $C_2$ and prolate or oblate depending on whether $C_0<0$ or $C_0>0$. Changing $C_1$ translates the sphere in ${\mathbb R}^3$ along the axis of symmetry and leaves the radii of curvature unchanged. 

The Hopf spheres are real-analytic iff $\mu=1+{\textstyle{\frac{1}{l+1}}}$ for some $l\in{\mathbb{N}}$ in which case we can carry out the double integration explicitly to find:
\[
r=C_2+C_1\cos\theta-2C_0\sum\limits_{k=0}^l\frac{(-1)^k}{(2k+1)(2k+2)}\binom{l}{k}\cos^{2k+2}\theta,
\]
as per Proposition \ref{p:quadsex1}.

To summarize
\vspace{0.1in}

\begin{Prop}\label{p:lHs}
For constants $\psi_\infty>0$, $\mu>1$ there is a 1-parameter family of rotationally symmetric spheres (unique up to translation) satisfying the linear Hopf relation (\ref{e:lihsdef}) and the radii of curvature are given by
\[
\psi=\psi_\infty-\mu C_0\sin^{\scriptstyle{\frac{2}{\mu-1}}}\theta
\qquad\qquad
s=C_0\sin^{\scriptstyle{\frac{2}{\mu-1}}}\theta,
\]
which are convex iff $\psi_\infty>C_0\mu$. They are real analytic iff
\[
\mu=1+{{\frac{1}{l+1}}} \qquad\qquad {\mbox{for some  }} l\in{\mathbb{N}}.
\]

\end{Prop}
\vspace{0.1in}


\section{Evolution by a Linear Hopf Relation}\label{s:3}

In this section we formulate the evolution equations of a rotationally symmetric sphere moving by a linear combination of its radii of curvature and give the complete solution in terms of initial data for integer values (\ref{e:hopfflow3}).

\vspace{0.1in}


\subsection{Evolution in RoC space}\label{s:3.1}
As in \cite{gak7}, a {\it classical curvature flow} of a sphere is a map $\vec{X}:S^2\times[0,t_1)\rightarrow {\mathbb R}^3$ such that
\begin{equation}\label{e:ccflow}
\frac{\partial \vec{X}}{\partial t}^\perp =-{{{\bf K}}}(\psi, s )\;\vec{n},
\qquad\qquad\qquad
\vec{X}(S^2,0)=S_0,
\end{equation}
where ${ {\bf K}}$ is a given function of the sum and difference of the radii of curvature, $\vec{n}$ is the unit normal vector to the 
flowing sphere and $S_0$ is an initial convex sphere. 

Consider the special case when the initial sphere is rotationally symmetric. Since the evolution equation, depending only on the second fundamental form, is rotationally symmetric, the spheres $S_t$ for $t\geq0$ will therefore also be rotationally symmetric. Moreover, the axis of symmetry remains stationary under the flow, and the north and south poles remain umbilic. 

In RoC space it as a curve flow on the plane subject to the boundary condition that it's end-points lie on the horizontal axis - the umbilic horizon. In fact, it is a {\it parameterized} curve and must satisfy the Codazzi-Mainardi relation (\ref{e:comain_rs}) which we have seen relates the vanishing of $s$ with the slope at the isolated umbilic points.

The curve evolves in RoC space under the linear Hopf flow as follows:

\vspace{0.1in}
\begin{Prop}\label{p:flow1}
The flow on RoC space for a curvature function ${\bf K}=\psi+\lambda s-\psi_\infty$ is
\begin{equation}\label{e:flow_psi}
\frac{\partial\psi}{\partial t}=\frac{\lambda-1}{2\sin\theta}\frac{\partial}{\partial \theta}\left(\sin\theta\frac{\partial\psi}{\partial \theta}\right)
-\lambda\cot\theta\frac{\partial\psi}{\partial \theta}
-\frac{2\lambda s}{\sin^2\theta} -\psi+\psi_\infty
\end{equation}
\begin{equation}\label{e:flow_s}
\frac{\partial s}{\partial t}=\frac{\lambda-1}{2\sin\theta}\frac{\partial}{\partial \theta}\left(\sin\theta\frac{\partial s}{\partial \theta}\right)
-\lambda\cot\theta\frac{\partial s }{\partial \theta} +\frac{1+\cos^2\theta}{\sin^2\theta} s.
\end{equation}
\end{Prop}
\begin{proof}
These equations can be obtained by the rotationally symmetric reduction of the general curvature flow equations given in Proposition 3 of \cite{gak7}. Alternatively, derive this ab initio as follows.

Start with the flow of the support function $r$, which for a curvature flow (\ref{e:ccflow}) is simply
\[
\frac{\partial r}{\partial t}=-{\bf K}=-\psi-\lambda s+\psi_\infty
\] 
and so, by equations (\ref{e:psi_def}) 
\begin{align}
\frac{\partial \psi}{\partial t}&=\frac{\partial }{\partial t}\left( r+\frac{1}{2\sin\theta}\frac{\partial}{\partial\theta}
\left(\sin\theta\frac{\partial r}{\partial\theta}\right) \right)\nonumber\\
&=-\psi-\lambda s+\psi_\infty-\frac{1}{2\sin\theta}\frac{\partial}{\partial\theta}\left(\sin\theta\frac{\partial}{\partial\theta}(\psi+\lambda s-\psi_\infty)\right)\nonumber\\
&=-\psi-\lambda s+\psi_\infty-{\textstyle{\frac{1}{2}}}\left(\frac{\partial^2\psi}{\partial \theta^2}
+\lambda\frac{\partial^2 s}{\partial \theta^2} \right)
-{\textstyle{\frac{1}{2}}}\cot\theta\left(\frac{\partial\psi}{\partial \theta}
+\lambda\frac{\partial s}{\partial \theta} \right)\nonumber.
\end{align}

By virtue of the Codazzi-Mainardi equation (\ref{e:comain_rs}) and its derivative, replace the derivatives of $ s$ by those of
$\psi$. The result is as stated.

A similar calculation gives the flow of the astigmatism $s$ as
\[
\frac{\partial s}{\partial t}={\textstyle{\frac{1}{2}}}\left(\frac{\partial^2\psi}{\partial \theta^2}
+\lambda\frac{\partial^2 s}{\partial \theta^2} \right)
-{\textstyle{\frac{1}{2}}}\cot\theta\left(\frac{\partial\psi}{\partial \theta}
+\lambda\frac{\partial s}{\partial \theta} \right),
\]
and by replacing the derivatives of $\psi$ by those of $s$ using the Codazzi-Mainardi equation (\ref{e:comain_rs}). The stated result follows.
\end{proof}
\vspace{0.1in}

When viewed in RoC space, this system is parabolic as long as $\lambda>1$. In fact, the second equation is entirely decoupled: the difference of the radii of curvature satisfies a linear reaction-diffusion equation with convection - with Dirichlet boundary conditions $s(0)=s(\pi)=0$.  

Our approach to solving the flow is to start by integrating this reaction-diffusion equation explicitly and to use quadratures to reconstruct the support function of the surface from the astigmatism $s$. Differentiation then yields $\psi$ as per equation (\ref{e:psi_def}).

It is easily seen that
\begin{Cor}\label{c:flow_cm}
Under the curvature flow the rotationally symmetric Codazzi-Mainardi equation (\ref{e:comain_rs}) is preserved:
\[
\frac{\partial}{\partial t}\left[\frac{\partial (\psi+ s )}{\partial \theta}+2\cot\theta\;s \right]=0.
\]
\end{Cor}

\vspace{0.1in}

In the next section the flow when $\lambda=1+{\textstyle{\frac{1}{n+1}}}$ will be completely integrated. An insight into how this is done is 
given by the flow of the following quantity:

\begin{Prop}\label{p:legpoly}
If $\lambda=1+{\textstyle{\frac{1}{n+1}}}$ then we have the flow
\[
\frac{\partial }{\partial t}\left(\frac{s}{\sin^{2+n}\theta}\right)
=\frac{1}{2(n+1)}\left[\frac{\partial^2}{\partial\theta^2}+\cot\theta\frac{\partial}{\partial\theta}
+\left((n+1)n-\frac{n^2}{\sin^2\theta}\right)\right]\frac{s}{\sin^{2+n}\theta}.
\]
\end{Prop}
In particular, if $n$ is a natural number, the eigen-space of the second order operator on the right-hand side consists of the associated Legendre polynomials $P_{l}^{n}(\theta)$ for $l\geq n$, and as we will see in the next section, these span the higher modes of the integer linear Hopf flow.

\vspace{0.1in}

\subsection{Solution for integer linear Hopf flows}\label{s:3.2}

In this section we completely solve the evolution of a smooth sphere by a linear combination of its radii of curvature in terms of initial conditions in the integer case where (\ref{e:hopfflow3}) holds. 

First introduce a decomposition of the $s$ in terms of trigonometric polynomials adapted to the flow.
As per Proposition \ref{p:decomp} the difference of the radii of curvature at time $t=0$ can be written
\[
s(0)=\sum_{l=0}^\infty(a_l+b_l\cos\theta)\sin^{2l+2}\theta,
\]
for constants $a_l,b_l$.

For any $n\in{\mathbb N}$ split this into
\begin{equation}\label{e:decomp0}
s(0)=\sum_{l=0}^{n-1}(a_l+b_l\cos\theta)\sin^{2l+2}\theta+\sum_{l=n}^\infty c_l\sin^{2+n}\theta P^n_l(\cos\theta),
\end{equation}
for constants $a_l,b_l,c_l$, where $P^n_l(\cos\theta)$ are associated Legendre polynomials.

\vspace{0.1in}
\begin{Thm}\label{t:5}
Given a decomposition (\ref{e:decomp0}) of an initial surface, the difference between the radii of curvature evolves under linear Hopf flow (\ref{e:hopfflow1}), (\ref{e:hopfflow2}) and (\ref{e:hopfflow3}) as follows:

\begin{align}
s=&\sum_{l=0}^{n-2}[A_l(t)+B_l(t)\cos\theta]\sin^{2l+2}\theta+(\tilde{a}_{n-1}e^{\mu_{n-1}t}
+\tilde{b}_{n-1}e^{\mu_{n-{\scriptstyle{\frac{1}{2}}}}t}\cos\theta)\sin^{2n}\theta\nonumber\\
&\qquad\qquad\qquad  +\sin^{2+n}\theta\sum_{l=n}^\infty c_lP^n_l(\cos\theta)e^{-\omega_lt},\nonumber
\end{align}
where
\begin{equation}\label{e:bigAt}
A_l(t)=\tilde{a}_le^{\mu_lt}+\sum_{m=1}^{n-l-1}(-1)^{n-l-m}{\textstyle{\frac{\prod\limits_{p=m}^{n-l-1}\nu_{n-p}}
{\prod\limits_{q=m+1}^{n-l}(\mu_{n-m}-\mu_{n-p})}}}\tilde{a}_{n-m}e^{\mu_{n-m}t}
\end{equation}
\begin{equation}\label{e:bigBt}
B_l(t)=\tilde{b}_le^{\mu_{l+{\scriptstyle{\frac{1}{2}}}}t}+\sum_{m=1}^{n-l-1}(-1)^{n-l-m}{\textstyle{\frac{\prod\limits_{p=m}^{n-l-1}\nu_{n-p}}
{\prod\limits_{q=m+1}^{n-l}(\mu_{n-m+{\scriptstyle{\frac{1}{2}}}}-\mu_{n-p+{\scriptstyle{\frac{1}{2}}}})}}}
\tilde{b}_{n-m}e^{\mu_{n-m+{\scriptstyle{\frac{1}{2}}}}t},
\end{equation}
for constants $\mu_l,\nu_l,\omega_l$ defined by
\[
\mu_l=\frac{(2l+1)(n-l)}{n+1} \qquad\qquad \nu_l=\frac{2l(n-l)}{n+1} \qquad\qquad \omega_l=\frac{l(l+1)-n(n+l)}{2(n+1)},
\]
and constants $\tilde{a}_l,\tilde{b}_l$ determined by the initial constants  ${a}_l,{b}_l$ via
\[
a_l=\tilde{a}_l+\sum_{m=1}^{n-l-1}(-1)^{n-l-m}{\textstyle{\frac{\prod\limits_{p=m}^{n-l-1}\nu_{n-p}}
{\prod\limits_{q=m+1}^{n-l}(\mu_{n-m}-\mu_{n-p})}}}\tilde{a}_{n-m},
\]
\[
b_l=\tilde{b}_l+\sum_{m=1}^{n-l-1}(-1)^{n-l-m}{\textstyle{\frac{\prod\limits_{p=m}^{n-l-1}\nu_{n-p}}
{\prod\limits_{q=m+1}^{n-l}(\mu_{n-m+{\scriptstyle{\frac{1}{2}}}}-\mu_{n-p+{\scriptstyle{\frac{1}{2}}}})}}}
\tilde{b}_{n-m},
\]
\end{Thm}
\vspace{0.1in}
\begin{proof}
The solution above can be checked directly, if laboriously, by substitution in the reaction-diffusion equation (\ref{e:flow_s}). However to 
illustrate the structure of the flow proceed as follows.

Consider Proposition \ref{p:legpoly}. Take a sum of associated Legendre polynomials
\begin{equation}\label{e:g1a}
s=\sin^{2+n}\theta\sum_{l=n}^\infty c_lP^n_l(\cos\theta),
\end{equation}
and allow the values of $c_l$ to vary in time, then the reaction-diffusion equation means that
\[
\frac{\partial}{\partial t}\left(\frac{s}{\sin^{2+n}\theta}\right)=\sum_{l=n}^\infty \frac{d c_l}{d t}P^n_l(\cos\theta)
=-\sum_{l=n}^\infty c_l\omega_lP^n_l(\cos\theta)
\]
for $\omega_l=\frac{l(l+1)-n(n+l)}{2(n+1)}$. Solving term by term yields
\[
s(t)=\sin^{2+n}\theta\sum_{l=n}^\infty \tilde{c}_lP^n_l(\cos\theta)e^{-\omega_lt},
\]
for constants $\tilde{c}_l$ and $l\geq n$. Thus we have solved for the stationary ($l=n$) and higher orders ($l>n$), which die off as $t\rightarrow\infty$. 

The sum (\ref{e:g1a}) only involves terms that are at least of order $2+2n$ and to lower this, the natural generalisation would be to use
Legendre functions for which $l<n$. Since these are less familiar, we compute the lower orders recursively
by hand.

First consider the following:

\vspace{0.1in}
\begin{Lem}\label{l:lap}
Let $\tilde{\triangle}$ be the elliptic operator on the right-hand side of equation (\ref{e:flow_s}):
\[
\tilde{\triangle}=\frac{\lambda-1}{2\sin\theta}\frac{\partial}{\partial \theta}\left(\sin\theta\frac{\partial }{\partial \theta}\right)
-\lambda\cot\theta\frac{\partial }{\partial \theta} +\frac{1+\cos^2\theta}{\sin^2\theta}.
\]
Then 
\[
\tilde{\triangle}\sin^{2m}\theta=\frac{m-n-1}{n+1}\left[(1-2m)\sin^{2m}\theta+2(m-1)\sin^{2m-2}\theta\right],
\]
and
\[
\tilde{\triangle}\cos\theta\sin^{2m}\theta=\frac{\cos\theta}{n+1}\left[m(2n-2m+1)\sin^{2m}\theta+2(m-1)(m-n-1)\sin^{2m-2}\theta\right].
\]
\end{Lem}
\vspace{0.1in}

Now flow the lower order in the decomposition:
\[
s=\sum_{l=0}^{n-1}(A_l+B_l\cos\theta)\sin^{2l+2}\theta,
\]
where $A_l$ and $B_l$ depend on time. By the previous Lemma the evolutions of $A_l$ and $B_l$ decouple and can be dealt with separately. 

Thus, consider first
\[
s=\sum_{l=0}^{n-1}A_l(t)\sin^{2l+2}\theta,
\]
and compute using Lemma \ref{l:lap}:
\begin{align}
\frac{\partial s}{\partial t}=&\sum_{l=0}^{n-1}\frac{d A_l}{d t}\sin^{2l+2}\theta=\tilde{\triangle}s
=\sum_{l=0}^{n-1}\frac{l-n}{n+1}A_l\left[-(2l+1)\sin^{2l+2}\theta+2l\sin^{2l}\theta\right]\nonumber\\
&={\textstyle{\frac{2n-1}{n+1}}}A_{n-1}\sin^{2n}\theta+\sum_{l=0}^{n-2}\left[{\textstyle{\frac{(n-l)(2l+1)}{n+1}}}A_l
-{\textstyle{\frac{(n-l-1)(2l+2)}{n+1}}}A_{l+1}\right]\sin^{2l+2}\theta.\nonumber
\end{align}
Comparing terms leads to the finite system of ODE's
\[
\frac{d }{d t}A_{n-1}={\textstyle{\frac{2n-1}{n+1}}}A_{n-1}
\qquad\qquad
\frac{d }{d t}A_l-\mu_lA_l=\nu_{l+1}A_{l+1}{\mbox{  for }}l<n-1,
\]
where
\[
\mu_l=\frac{(2l+1)(n-l)}{n+1}
\qquad\qquad
\nu_l=\frac{2l(n-l)}{n+1}.
\]
These can be integrated sequentially starting with $l=n-1$ down to $l=0$. The effect of each step is the introduction of a new exponential term
of coefficient $\mu_l>0$, together with terms containing all lower exponents. The closed form of this is stated in the Theorem.

The case where
\[
s=\sum_{l=0}^{n-1}B_l(t)\cos\theta\sin^{2l+2}\theta,
\]
is similar, the key steps being
\begin{align}
\frac{\partial s}{\partial t}=&\sum_{l=0}^{n-1}\frac{d B_l}{d t}\cos\theta\sin^{2l+2}\theta
=\sum_{l=0}^{n-1}B_l\left[{\textstyle{\frac{(2n-2l-1)(l+1)}{n+1}}}\sin^{2l+2}\theta
+{\textstyle{\frac{2l(l-n)}{n+1}}}\sin^{2l}\theta\right]\cos\theta\nonumber\\
&={\textstyle{\frac{n}{n+1}}}B_{n-1}\cos\theta\sin^{2n}\theta+\sum_{l=0}^{n-2}{\textstyle{\frac{l+1}{n+1}}}\left[{\textstyle{(2n-2l-1)}}B_l
-{\textstyle{2(2n-2l-1)}}B_{l+1}\right]\cos\theta\sin^{2l+2}\theta,\nonumber
\end{align}
yielding the ODE's
\[
\frac{d }{d t}B_{n-1}={\textstyle{\frac{n}{n+1}}}B_{n-1}
\qquad\qquad
\frac{d }{d t}B_l-\mu_{l+{\scriptstyle{\frac{1}{2}}}}B_l=\nu_{l+1}B_{l+1}{\mbox{  for }}l<n-1,
\]
with solutions as stated.
\end{proof}
\vspace{0,1in}
To completely solve the problem, use quadratures to reconstruct the flowing surface. 

\vspace{0.1in}
\begin{Thm}\label{t:6}
The support function evolves by
\begin{align}
r&=\psi_\infty+(D_2-D_1\cos\theta)e^{-t}-\sum_{l=0}^{n-1}\sum_{k=0}^lA_l(t)\frac{2(-1)^k}{(2k+1)(2k+2)}\binom{l}{k}\cos^{2k+2}\theta\nonumber\\
&\qquad-\sum_{l=0}^{n-1}\sum_{k=0}^{l+1}B_l(t)\frac{(-1)^k}{(l+1)(2k+1)}\binom{l+1}{k}\cos^{2k+1}\theta\nonumber\\
&\qquad\qquad-\sum_{l=n+2}^{\infty}\frac{2c_l}{(l+n+2)(l+n+1)(l-n)(l-n-1)}\sin^{2+n}\theta P^{2+n}_l(\cos\theta)e^{-\omega_lt}.\nonumber
\end{align}
with $A_l(t)$ and $B_l(t)$ given by equations (\ref{e:bigAt}) and (\ref{e:bigBt}), and constants $D_1$ and $D_2$ determined by the initial sphere.
\end{Thm}
\begin{proof}
This follows from Propositions \ref{p:quadsex1} and \ref{p:quadsex2}, which integrate $s$ twice to get $r$ and differentiates twice to get $\psi$. Here the "constants" now depend upon time, as described in Theorem \ref{t:5}, and the two "constants" of integration, $C_1(t)$ and $C_2(t)$ must ensure that
\[
\frac{\partial r}{\partial t}=-{\mathbb K}=-\psi-(1+{\textstyle{\frac{1}{n+1}}})s+\psi_\infty.
\]
This implies that we must have 
\[
C_1=D_1e^{-t} \qquad\qquad C_2=\psi_\infty+D_2e^{-t},
\]
for constants $D_1,D_2$. 

This completes the proof of the Theorem.

\end{proof}

Theorems \ref{t:1}, \ref{t:2} and \ref{t:3} of the Introduction are consequences of from Theorems \ref{t:4}, \ref{t:5} and \ref{t:6}. 

To see this, note that Theorem \ref{t:1} follows from Theorem \ref{t:5}, since the astigmatism $s$ either goes to zero, to that of a non-round integer linear Hopf sphere, or to infinity, depending on whether the order $k$ of the initial surface is greater than, equal to or less than $n$, respectively.

Theorem \ref{t:2} then comes from combining Theorem \ref{t:1}  with Theorem \ref{t:4} in the non-degenerate case. Finally, Theorem \ref{t:3} is the special case of Theorem \ref{t:1} with $n=0$.

\vspace{0.1in}

\section{Properties of the Solutions}\label{s:4}

In this section we explore some of the different types of behaviour exhibited by the solutions presented in Theorems \ref{t:5} and \ref{t:6}.

\vspace{0.1in}

\subsection{A finite class of sufficient generality}\label{s:4.2}

Let us explore a low degree polynomial example in detail. That is, consider an initial surface $S_0$ such that the initial astigmatism is
\begin{equation}\label{e:ex00}
s(0)=(a_0+b_0\cos\theta)\sin^2\theta+(a_1+b_1\cos\theta)\sin^4\theta,
\end{equation}
for constants $a_0,b_0,a_1,b_1$. We will solve the $n=0$ and $n=1$ integer linear Hopf flow for this initial data in detail.

First, Theorem \ref{t:5} gives the solution to the reaction-diffusion equation with $n=0$ and these initial conditions as
\begin{equation}\label{e:flowexs0}
s=[a_0+{\textstyle{\frac{2}{3}}}a_1(1-e^{-3t})+(b_0e^{-t}+{\textstyle{\frac{2}{5}}}b_1(e^{-t}-e^{-6t}))\cos\theta]\sin^2\theta
+(a_1e^{-3t}+b_1e^{-6t}\cos\theta)\sin^4\theta.
\end{equation}
Now Propositions \ref{p:quadsex1} and \ref{p:quadsex2} yields a support function
\begin{align}
r=&\psi_\infty-{\textstyle{\frac{4}{3}}}a_1-2a_0+(c_0-{\textstyle{\frac{2}{3}}}b_0-{\textstyle{\frac{4}{15}}}b_1)e^{-t}
+[2a_0+{\textstyle{\frac{4}{3}}}a_1+c_1+({\textstyle{\frac{2}{3}}}b_0+{\textstyle{\frac{4}{15}}}b_1)e^{-t}]\cos\theta\nonumber\\
&\qquad+{\textstyle{\frac{2}{15}}}[15a_0+10a_1+(5b_0+2b_1)e^{-t}\cos\theta]\sin^2\theta
+{\textstyle{\frac{4}{5}}}[5a_1e^{-3t}+3b_1e^{-6t}\cos\theta]\sin^4\theta,\nonumber
\end{align}
where $c_0$ and $c_1$ are determined by the initial support function. The surface in ${\mathbb R}^3$ parameterized by its Gauss coordinates can be reconstructed via its profile curve in  ${\mathbb R}^2$ (any plane containing the axis of symmetry)
\[
x^1=r\cos\theta -\sin\theta\frac{dr}{d\theta} 
\qquad\qquad
x^2=r\sin\theta+\cos\theta\frac{dr}{d\theta} .
\] 
The mean radius of curvature works out to be
\begin{align}\label{e:flowexpsi0}
\psi=&\psi_\infty+[c_0+({\textstyle{\frac{2}{3}}}b_0+{\textstyle{\frac{4}{15}}}b_1)(\cos\theta-1)]e^{-t}\nonumber\\
&\qquad+[-2a_0+{\textstyle{\frac{4}{3}}}(e^{-3t}-1)a_1+(-{\textstyle{\frac{5}{3}}}b_0e^{-t}+{\textstyle{\frac{2}{15}}}b_1(6e^{-6t}-5e^{-t}))\cos\theta]
\sin^2\theta\nonumber\\
&\qquad\qquad\qquad+{\textstyle{\frac{1}{10}}}[15a_1e^{-3t}+14b_1e^{-6t}\cos\theta]\sin^4\theta.
\end{align}

The behaviour of the integer linear Hopf flow with $n=0$ ($\lambda=2$) on RoC space can be studied via equations (\ref{e:flowexs0}) and (\ref{e:flowexpsi0}) for different initial parameters $a_0,b_0,a_1,b_1,c_0$.

For this class of solutions clearly we have

\vspace{0.1in}
\begin{Prop}
Under linear Hopf flow with $n=0$ the initial surfaces with (\ref{e:ex00}) converge to non-round Hopf spheres iff $a_0\neq0$. If $a_0=0$ they converge to a round sphere of radius $\psi_\infty$.

\end{Prop}
\vspace{0.1in}

Secondly, consider integer linear Hopf flow with $n=1$. Once again Theorem \ref{t:5} gives the solution for $s$, which is
\begin{equation}\label{e:flowexs1}
s=(a_0+b_0\cos\theta)e^{t/2}\sin^2\theta+(a_1+b_1e^{-t}\cos\theta)\sin^4\theta.
\end{equation}

Now Propositions \ref{p:quadsex1} and \ref{p:quadsex1} yield a support function
\begin{align}
r=&\psi_\infty-{\textstyle{\frac{4}{3}}}a_1-2a_0+(c_0-{\textstyle{\frac{2}{3}}}b_0-{\textstyle{\frac{4}{15}}}b_1)e^{-t}
+[2a_0+{\textstyle{\frac{4}{3}}}a_1+c_1+({\textstyle{\frac{2}{3}}}b_0+{\textstyle{\frac{4}{15}}}b_1)e^{-t}]\cos\theta\nonumber\\
&\qquad+{\textstyle{\frac{2}{15}}}[15a_0+10a_1+(5b_0+2b_1)e^{-t}\cos\theta]\sin^2\theta
+{\textstyle{\frac{4}{5}}}[5a_1e^{-3t}+3b_1e^{-6t}\cos\theta]\sin^4\theta,\nonumber
\end{align}
with mean radius of curvature 
\begin{align}\label{e:flowexpsi1}
\psi=&\psi_\infty+(c_0+({\textstyle{\frac{2}{3}}}b_0+{\textstyle{\frac{4}{15}}}b_1)(\cos\theta-1))e^{-t}\nonumber\\
&\qquad+[-2a_0+{\textstyle{\frac{4}{3}}}(e^{-3t}-1)a_1+(-{\textstyle{\frac{5}{3}}}b_0e^{-t}+{\textstyle{\frac{2}{15}}}b_1(6e^{-6t}-5e^{-t}))\cos\theta]
\sin^2\theta\nonumber\\
&\qquad\qquad\qquad+{\textstyle{\frac{1}{10}}}[15a_1e^{-3t}+14b_1e^{-6t}\cos\theta]\sin^4\theta.
\end{align}

\vspace{0.1in}
\begin{Prop}
Under linear Hopf flow with $n=1$ the initial surfaces with (\ref{e:ex00}) diverges to infinity unless both $a_0=0$ and $b_0=0$, in which
case it converges to a non-round Hopf sphere if $a_1\neq0$, and to a round sphere of radius $\psi_\infty$ if $a_1=0$. 

\end{Prop}
\vspace{0.1in}

Note that when the RoC diagram diverges, the associated surface in ${\mathbb R}^3$ can pull itself inside out by passing through its focal set. These are the points where the RoC diagram intersects the diagonals on RoC space.

\vspace{0.1in}

\subsection{Jumping slope at an isolated umbilic point}\label{s:4.3}

The Codazzi Mainardi relation relates the order of vanishing of $s$ at the isolated umbilic points with the slope of the Weingarten relation. For smooth surfaces this slope is quantized and therefore can only change under the flow we are considering by jumping. 

Consider the evolution with $n=0$ and $\psi_\infty=10$ of the RoC diagram for an initial sphere with  $(a_0,b_0,a_1,b_1,c_0)=(1,1,2,3,1)$. Since $a_0=b_0\neq0$ this surface is degenerate of order 0. The RoC diagram is a simple bow above the horizontal axis, so the sphere is oblate. 

Initially, the slopes at the two umbilic points are not equal, namely
\[
\mu_N(0)=2 \qquad\qquad \mu_S(0)={\textstyle{\frac{3}{2}}}.
\]
By the solution we have given in the previous subsection, we see in Figure 2 that the bow-shaped RoC diagram converges to a line segment on the plane and that  for $t>0$ 
\[
\mu_N(t)=\mu_S(t)=2.
\]
Thus the slope at the south pole umbilic jumps initially.

The flow converges exponentially in Euclidean 3-space to a non-round linear Hopf sphere satisfying $\psi+2s=10$.

\vspace{0.1in}

\subsection{Contracting an umbilic circle}\label{s:4.4}

Consider the evolution with $n=0$ and $\psi_\infty=10$ of the initial sphere with $(a_0,b_0,a_1,b_1,c_0)=(1,5,2,3,1)$. This initial sphere is turnip-shaped: part prolate, part oblate, separated by an umbilic circle. 

As can be seen in Figure 3, under the flow  the prolate part shrinks as the umbilic circle contracts to the south pole. It reaches the pole in finite time and the surface is degenerate for an instant - the slope jumps instantaneously, returning to its original value as time continues. 

Thus the umbilic circle disappears with a ``pop'' in finite time. Again the flow converges to a non-round linear Hopf sphere satisfying $\psi+2s=10$, as predicted by Theorem \ref{t:2}.

To get diverging solutions consider the previous solution given by equations (\ref{e:flowexs1}) and (\ref{e:flowexpsi1}) with $n=1$ and $\psi_\infty=10$. If we start with an initial sphere such as the one generated by $(a_0,b_0,a_1,b_1,c_0)=(2,5,1,-1,1)$, this yields an exponentially increasing flow, as seen in Figure 4.

Such a flow will of necessity lose its convexity when it crosses either diagonal in the RoC plane. This can be understood in ${\mathbb R}^3$ in terms of the focal set\cite{gak3}. 

\vspace{0.1in}

\subsection{Solitons in RoC space}\label{s:4.1}

Consider flow by a linear curvature function for an initial surface $S_0$ which is itself a linear Hopf sphere and ask: when does the flowing surface stay within the class of linear Hopf spheres?

By Proposition \ref{p:lHs}, the radii of curvature of a linear Hopf sphere with slope  $\mu$ are 
\[
\psi=C_0-\mu C_1\sin^{\scriptstyle{\frac{2}{\mu-1}}}\theta
\qquad\qquad
s=C_1\sin^{\scriptstyle{\frac{2}{\mu-1}}}\theta.
\]

We therefore seek solutions to the flow equations (\ref{e:flow_psi}) and (\ref{e:flow_s}) of the above form with $C_0,C_1$ and $\mu$ depending 
on time $t$. Substituting these in and adding equation (\ref{e:flow_psi}) and $\mu$ times equation (\ref{e:flow_s}) yields
\[
\frac{dC_0}{dt}+C_0-\psi_\infty-C_1\sin^{\scriptstyle{\frac{2}{\mu-1}}}\theta\left(\frac{d\mu}{dt}+\frac{2(\mu-\lambda)}{\sin^2\theta}\right)=0,
\]
Clearly this can only hold for all $\theta$ iff either $C_1=0$ (round initial sphere) or $\mu=\lambda$ or $\mu=2$ for all time. 

In the first two cases, the sphere converges exponentially through parallel spheres to the appropriate linear Hopf sphere. Such motion is by an isometry of the hyperbolic/AdS metric.

Assume now that $\mu=2$ for all time. Then the equation (\ref{e:flow_s}) is
\[
\frac{dC_1}{dt}+(\lambda-2)C_1=0,
\]
which integrates to  $C_1=C_2e^{(2-\lambda)t}$ for some constant $C_2$.

Now equation (\ref{e:flow_psi}) becomes
\[
\frac{dC_0}{dt}+C_0-\psi_\infty+2C_2(\lambda-2)e^{(2-\lambda)t}=0,
\]
with solution for $\lambda\neq3$ 
\[
C_0=\psi_\infty+C_3e^{-t}+2C_2\left(\frac{\lambda-2}{\lambda-3}\right)e^{(2-\lambda)t},
\]
for some constant $C_3$, and for $\lambda=3$ 
\[
C_0=\psi_\infty+(C_3-2C_2t)e^{-t}.
\]

For each linear flow (fixed $\lambda$ and $\psi_\infty$) there is a 2-parameter family of solutions to the initial boundary value 
problem in RoC space. The 2 parameters $C_2$ and $C_3$ are fixed by the initial choice of linear Hopf sphere with $\mu=2$. 

Suppose that the initial surface has
\[
\psi=\psi_0-2s_{\scriptstyle{\frac{\pi}{2}}}\sin^2\theta
\qquad\qquad
s=s_{\scriptstyle{\frac{\pi}{2}}}\sin^2\theta,
\]
so that the constants are $\psi_0=\psi(0,0)$ and $s_{\scriptstyle{\frac{\pi}{2}}}=s({\textstyle{\frac{\pi}{2}}},0)$.

To summarize, the previous calculation and its integral in terms of the support function yield:

\vspace{0.1in}
\begin{Prop}\label{p:soliton}
Given an initial linear Hopf surface $S_0$ with $\mu=2$, the linear flow (\ref{e:hopfflow1}) and (\ref{e:hopfflow2}) has a unique solution:

When $\lambda\neq 3$ the support function is
\[
r=\psi_\infty +\left(\psi_0-\psi_\infty-{\textstyle{\frac{2(\lambda-2)}{\lambda-3}}}s_{\scriptstyle{\frac{\pi}{2}}}\right)e^{-t}
+{\textstyle{\frac{1}{2}}}s_{\scriptstyle{\frac{\pi}{2}}}\left({\textstyle{\frac{\lambda+1}{\lambda-3}}}-\cos 2\theta\right)e^{(2-\lambda)t},
\]
and the radii of curvature are
\[
\psi=\psi_\infty +\left(\psi_0-\psi_\infty-{\textstyle{\frac{2(\lambda-2)}{\lambda-3}}}s_{\scriptstyle{\frac{\pi}{2}}}\right)e^{-t}
+2s_{\scriptstyle{\frac{\pi}{2}}}\left({\textstyle{\frac{\lambda-2}{\lambda-3}}}-\sin^2\theta\right)e^{(2-\lambda)t},
\]
\[
s=s_{\scriptstyle{\frac{\pi}{2}}}\sin^2\theta e^{(2-\lambda)t}.
\]
When $\lambda= 3$ the support function is
\[
r=\psi_\infty +\left[\psi_0-\psi_\infty-s_{\scriptstyle{\frac{\pi}{2}}}(2(t+1)-\sin^2\theta)\right]e^{-t},
\]
and the radii of curvature are
\[
\psi=\psi_\infty +\left[\psi_0-\psi_\infty-2s_{\scriptstyle{\frac{\pi}{2}}}(t+\sin^2\theta)\right]e^{-t}
\qquad\qquad
s=s_{\scriptstyle{\frac{\pi}{2}}}\sin^2\theta e^{-t}.
\]
\end{Prop}
\vspace{0.1in}

The natural geometry on RoC space is two copies of the hyperbolic/AdS half-plane joined along their boundaries (the umbilic horizon). Their common isometries acting on an RoC diagram correspond to dilations about the origin in ${\mathbb R}^3$ and moving 
to a parallel surface.

Associated solitons are

\vspace{0.1in}
\begin{Prop}
A linear Hopf sphere evolving by a linear flow with $\mu=\lambda$ moves by translation in RoC space.

A linear Hopf sphere with $\mu=2$ evolving by a linear flow with $\lambda\neq3$ and 
$\psi_\infty=\psi_0-{\textstyle{\frac{2(\lambda-2)}{\lambda-3}}}$ moves by a dilation about the point $(\psi_\infty,0)$ on the umbilic horizon. The dilation is contracting for $\lambda>2$ and expanding for $\lambda<2$.
\end{Prop}

\vspace{0.1in} 

The soliton with $\psi_\infty=10, \lambda=4, s_{\scriptstyle{\frac{\pi}{2}}}=1$ and $\psi_0=\psi_\infty+{\textstyle{\frac{2(\lambda-2)}
{\lambda-3}}}$ is plotted in Figure 5. This is converging to a sphere of radius 10. The direction of flow is indicated by the box arrows here and throughout.

\includegraphics[width=140mm]{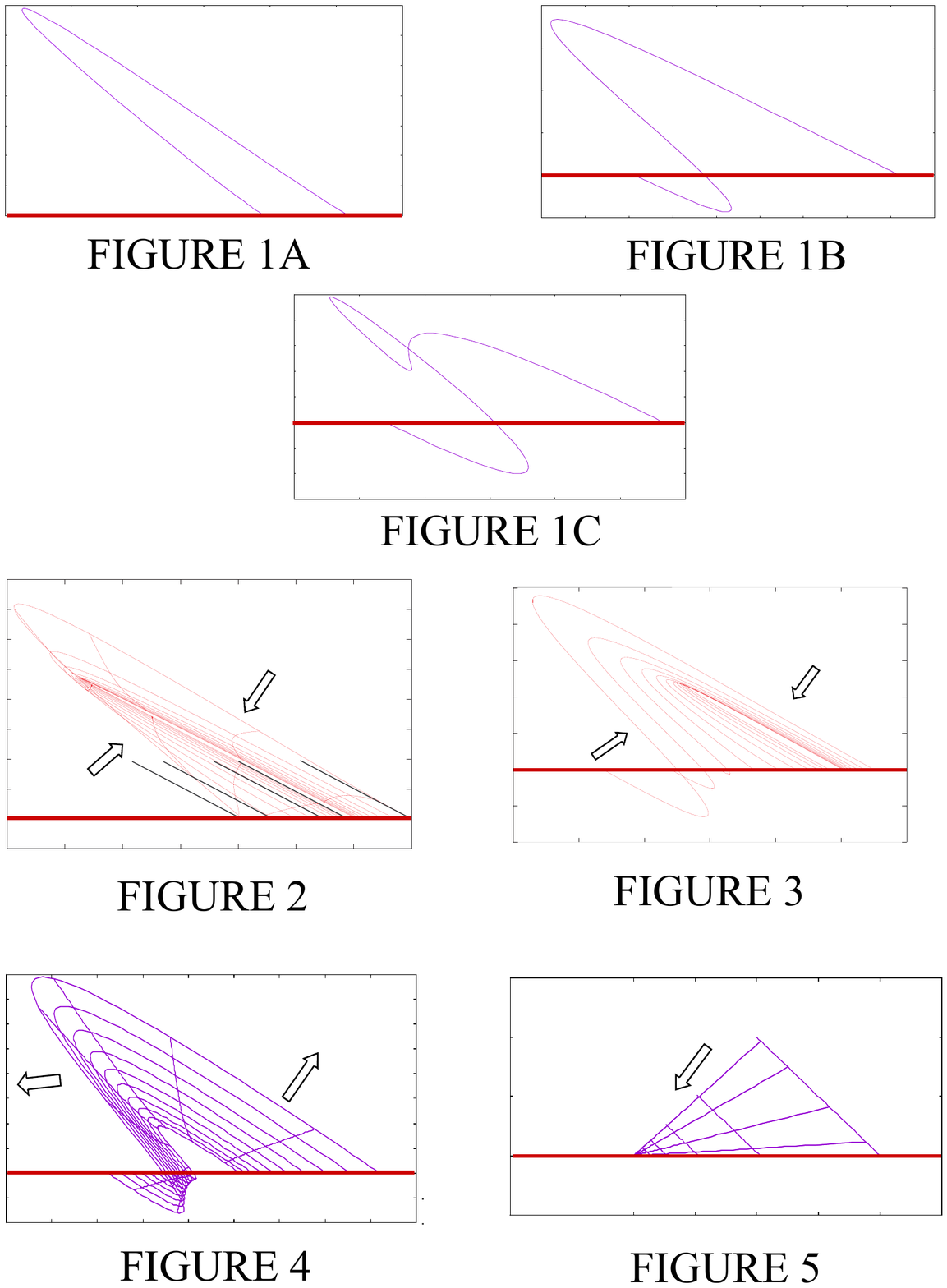}


\end{document}